\documentclass[11pt]{amsart}
\usepackage[active]{srcltx}
\usepackage{a4wide}
\usepackage{amsthm,amsfonts,amsmath,mathrsfs,amssymb}
\usepackage{dsfont}
\usepackage[T1]{fontenc}
\usepackage[utf8]{inputenc}
\usepackage{fouriernc}

\usepackage{hyperref}

\usepackage[shortlabels]{enumitem}
\usepackage[dvipsnames,usenames]{color}

\usepackage{comment}

\def\a{\alpha}		 \def\b{\beta}		  
		 	  
\def\la{\lambda}	 \def\om{\omega}	  \def\O{\Omega}
	 		  \def\t{\theta}
		 	  
\def\z{\zeta}		 \def\F{\Phi}		  

	 \def\C{{\mathbb C}}
\def\D{{\mathbb D}}	 
	 
	 \def\R{{\mathbb R}}
\def\T{{\mathbb T}}	 \def\Z{{\mathbb Z}}

	 \def\cb{{\mathcal B}}
\def\cc{{\mathcal C}}

	 \def\cp{{\mathcal P}}	
\def\rad{{\mathcal R}}

\def\({\left(}		 \def\){\right)}

\DeclareMathOperator{\re}{Re}

\newtheorem{thm}{\scshape Theorem}
\newtheorem{lem}[thm]{\scshape Lemma}
\newtheorem{prop}[thm]{\scshape Proposition}
\newtheorem{cor}[thm]{\scshape Corollary}
\newtheorem{ex}[thm]{\scshape Example}
\newtheorem*{conj}{\scshape Conjecture}
\newtheorem{other}{\sc Theorem}				 

\theoremstyle{definition}
\newtheorem{rem}[thm]{\scshape Remark}
\theoremstyle{plain}
\begin{document}

\title{On extremal problems related to the Krzy\.{z} conjecture}

\author[S. M. Instanes]{Sarah May Instanes}
\address{Department of Mathematical Sciences, Norwegian University of Science and Technology, NO-7491 Trondheim, Norway}
\email{sarah.m.instanes@ntnu.no}

\author[A. Llinares]{Adri\'an Llinares}
\address{Departamento de Matem\'aticas, Universidad Aut\'onoma de Madrid, 28049 Madrid, Spain}
\email{adrian.llinares@uam.es}
\urladdr{https://sites.google.com/view/adrian-llinares/home}

\author[D. Vukoti\'c]{Dragan Vukoti\'c}
\address{Departamento de Matem\'aticas, Universidad Aut\'onoma de Madrid, 28049 Madrid, Spain}
\email{dragan.vukotic@uam.es}
\urladdr{http://www.uam.es/dragan.vukotic}

\date{October 18, 2026}


\begin{abstract}
We study general extremal problems similar or closely related to the Krzy\.z conjecture and establish the structural properties of extremal functions, pointing out some differences with the original Krzy\.z case. We also introduce a family of more specific problems of this kind, interpolating between the Krzy\.z problem and the problem of maximizing the constant term, and show that they are closely related to the Krzy\.z conjecture: the conjecture implies the analogous statement for all of them and, conversely, if one of them has the same extremal value as the Krzy\.z problem, then this value is $2/e$. This may represent a new approach to attacking the original problem. Finally, we show that for these problems an extremal function whose Taylor coefficients $a_1,\ldots,a_{\lfloor n/3\rfloor}$ vanish must be the conjectured one; as a special case, we obtain a new proof of a recent result of MacDonald.
\end{abstract}
\maketitle

\section{Introduction}
 \label{sect:intro}
\par
We study extremal problems for bounded non-vanishing analytic functions in the unit disc $\D$ that generalize the well-known Krzy\.z problem. Denote by $\cb_\ast$  the class of all analytic functions $f$ in the unit disc $\D$ such that $0<|f(z)|\le 1$ for all $z$ in $\D$. If $f(z)= \sum_{j=0}^\infty a_j z^j$ in $\D$, let
\begin{equation}
 M_n = \sup \left\{ \re a_n\,\colon\,f\in\cb_\ast \right\}, \quad n\ge 1.
 \label{eq:extr-ref} 									
\end{equation}
Standard arguments involving normal families show that the supremum is attained for each $n\ge 1$, hence an extremal function always exists. Since instead of $f(z)$ we can consider $\mu f(\la z)$ with $|\la|= |\mu|=1$, after such rotations we may assume from the start that both $a_0$ and $a_n$ are real and positive. The celebrated \textit{Krzy\.z conjecture\/}, posed in 1968 in \cite{Kr}, states that $M_n=2/e$ and the only extremal function normalized in this way is
\begin{equation}
f_n(z)= e^{\frac{z^n-1}{z^n+1}} = \frac{1}{e} + \frac{2}{e} z^n + \ldots
 \label{eq:extr-fcn} 									
\end{equation}
This problem provides the main inspiration for our study and is still open. Upper bounds, valid for all $n\ge 1$ simultaneously and strictly smaller than $1$ (with approximate values of $0.999\ldots$), were obtained in \cite{Ho} and \cite{Er}. The conjecture has been proved for $n\le 5$ (\textit{cf.\/} \cite{Kr}, \cite{HSZ}, \cite{Br2}, \cite{Ta}, \cite{Sz} and \cite{Sa}), including a recent  computer-assisted proof \cite{W} for $n=6$, but remains open to this date for all larger values of $n$. For further information and some closely related questions, we mention \cite{HSZ}, \cite{LS}, \cite{Sz}, \cite{MSUV}, and \cite{AM}, as well as the recent preprints \cite{LZ} and \cite{MD}, and \cite{ELMSUV}.
\par
We will use the following known facts about normalized extremal functions, which we collect here for convenience. The three parts of the statement can be found respectively in \cite{HSZ}, in both \cite{HSZ} and \cite{MSUV}, and in \cite[Theorem~1]{MSUV}.
\begin{other} \label{thm:A}
Let $n\ge 1$ and let $f(z)=\sum_{j=0}^\infty a_j z^j$ be a normalized extremal function for \eqref{eq:extr-ref}, that is, $a_0>0$ and $a_n=M_n$. Then:
		\begin{enumerate}[(i)]
\item \label{thm:Ai}
    $f(z)=\exp\(\sum_{j=1}^N r_j\frac{\a_j z+1}{\a_j z-1}\)$ for some $N$ with $1\le N\le n$, some $r_j>0$, and distinct points $\a_j$ with $|\a_j|=1$;
\item \label{thm:Aii} $a_n\ge 2a_0$;
\item \label{thm:Aiii} $a_n=2a_0$ if and only if $M_n=2/e$ and $f=f_n$.
\end{enumerate}
\end{other}

Our purpose is to consider general coefficient extremal problems for non-vanishing bounded analytic functions. A useful survey of extremal  problems of this kind is \cite{BK}. A comprehensive treatment of various linear and non-linear coefficient extremal problems can be found in the classical reference \cite{KhS}.
\par
Let $f(z)=\sum_{j=0}^\infty a_j z^j$ be analytic in $\D$. If $|f(z)|\le 1$ for all $z\in\D$, the Schwarz--Pick lemma yields $|a_1|\le 1-|a_0|^2$. The F.~Wiener transform $W_n f$ of $f$ \cite[Section~1.7]{MSUV} is defined by
\begin{equation}
 W_n f(z) = \frac{1}{n} \sum_{j=0}^{n-1} f\( e^{(2\pi j i)/n} z\) = \sum_{j=0}^\infty a_{j n} z^{j n} = G(z^n),
  \label{eq:wiener} 	
\end{equation}
where $G$ is also analytic in $\D$ and bounded by $1$. Since $G(0)=a_0$ and $G^\prime (0)=a_n$, the Schwarz--Pick lemma applied to $G$ yields $|a_n|\le 1-|a_0|^2$ for all $n\ge 1$. Moreover, if $f\in\cb_\ast$, then clearly $|a_n|<1$ for all $n\ge 1$ and $|a_0|=1$ if and only if $f$ is a unimodular constant. In view of this, if we assume that $h: \overline{\D}^{n+1} \to \C$ is holomorphic in $\D^{n+1}$, continuous in $\overline{\D}^{n+1}$, and non-constant, then the coefficient functional of the form $\F (f) = h(a_0,a_1,\ldots,a_n)$ is well defined for every $f\in\cb_\ast$. We study the associated extremal problem of finding
\begin{equation}
 F_n = \sup \left\{ \re \F (f)\,\colon\, f\in\cb_\ast \right\}.
 \label{eq:extr-alt} 									
\end{equation}
The Krzy\.z problem corresponds to the special case $h(a_0,a_1,\ldots,a_n)=a_n$. Our Theorem~\ref{thm:structure}, formulated in Section~\ref{sect:gen-probl-struct}, shows that the maximum in \eqref{eq:extr-alt} is always attained and, moreover, every extremal function is either a unimodular constant or a singular inner function with at most $n$ atoms; that is, a unimodular multiple of a function like $f$ in Theorem~\ref{thm:A}~\ref{thm:Ai}. This has a flavor of other known results but we have not been able to find a complete proof of the result in this generality in the literature. We also show that if the function $h$ defining the functional $\Phi$ does not depend on the variable $a_0$, then all extremal functions are non-constant and, thus, have the same structure as the ones for the original Krzy\.z problem. On the other hand, we give specific non-trivial examples (where $\F$ also depends on other variables) showing that an extremal function can actually be constant. It should be mentioned that, for a fixed value of $\log a_0$, Romanova \cite{R} studied the extremal problem for bounded non-vanishing analytic functions for the linear functional of the form $\F (f) = c_1 a_1 + \ldots + c_n a_n$ and derived a similar structure using an extremal problem for functions with positive real part and Loewner chains.

We also discuss some issues related to normalization, since for certain problems we cannot freely assume that $f(0)>0$ because of the risk of losing the structure proved in Theorem~\ref{thm:structure}. Examples are given to illustrate some differences in such cases. For a smaller class of functionals $\Phi$ that only depend on the variables $a_0$ and $a_n$, in Section~\ref{sect:gen-probl-struct} we show that if a normalized extremal function is unique, it must be a function of $z^n$ (Theorem~\ref{thm:Uniqueness}).

This last situation, together with some known results for the Krzy\.z conjecture, motivates our Section~\ref{sect:spec-probl}, where we study a special family of such functionals $\F$ depending only on the coefficients $a_0$ and $a_n$ of $f$. In view of parts (ii) and (iii) of Theorem~\ref{thm:A}, it seems quite appropriate to study convex combinations of the functionals $\F_0 (f) = a_n$ and $\F_1 (f) = 2 a_0$. It also makes sense to restrict our attention to a smaller class containing all candidates for extremal functions for the Krzy\.z problem. For $n\ge 1$ we define $\cb_n$ as the set of all functions $f\in\cb_\ast$ which are singular inner functions with at most $n$ atoms and for which $a_n\ge 2 a_0>0$.

Next, let $p$ be a holomorphic function on $\D^2$ such that for $x$, $y$ real and $0<2x<y<1$, the value $p(x,y)$ is real with $0<p(x,y)\le 1$. By continuity, for $0<x<1/2$ we have $0\le p(x,2x)\le 1$. Constant functions $p$ are allowed. Under these restrictions on $p$, consider the functionals
\begin{equation}
 \F_p (f) = 2 p(a_0,a_n)\,a_0 + (1-p(a_0,a_n))\,a_n, \qquad f\in \cb_n,
 \label{eq:spec-funct} 									
\end{equation}
where $a_0$ and $a_n$ are the usual coefficients of the function $f$  and consider
\begin{equation}
		G_n = G_n(p) = \sup \left\{ \re \F_p (f)\,\colon\,f\in {\cb_n} \right\},
		\label{eq:extr-alt-p} 						
\end{equation}
(We write $G_n(p)$ when the dependence on $p$ needs to be emphasized.) This is of interest since it forces the inequalities $2a_0\le \F_p (f)\le a_n$ for every $f\in \cb_n$; moreover, $\F_p (f) = a_n$ holds if and only if $a_n=2 a_0$, which is precisely the condition known to imply $f=f_n$ for an extremal function $f$ of the Krzy\.z problem. The following result from Section~\ref{sect:spec-probl} connects the problems \eqref{eq:extr-ref} and \eqref{eq:extr-alt-p}.
    \begin{thm} \label{thm:rel-probl}
Let $n\ge 1$ and let $\F_p$ be an admissible functional of the form \eqref{eq:spec-funct}.
	\begin{enumerate}[(a)]
\item \label{thm:rel-probl:a}
If the function $f_n(z)= e^{\frac{z^n-1}{z^n+1}}$ is extremal for the Krzy\.z problem \eqref{eq:extr-ref}, that is, if $M_n=2/e$, then $G_n(p)=2/e$ and $f_n$ is the only extremal function for \eqref{eq:extr-alt-p}.
\item \label{thm:rel-probl:b}
If $G_n(p)=M_n$, then $M_n=2/e=G_n(p)$, the function $f_n$ is extremal for \eqref{eq:extr-ref}, and it is the only extremal function for \eqref{eq:extr-alt-p}. In particular, if
$$
 \sup \left\{ 2 a_0\,\colon\,f\in\cb_n \right\} = \sup \left\{ a_n\,\colon\,f\in\cb_n \right\},
$$
then $M_n=2/e$ and $f_n$ is extremal for the Krzy\.z problem \eqref{eq:extr-ref}.
\item \label{thm:rel-probl:c}
If there is only one extremal function for \eqref{eq:extr-alt-p}, then $G_n(p)=2/e$ and this function is $f_n$.
\end{enumerate}
    \end{thm}
Note that the original Krzy\.z problem can be viewed as the limit case of \eqref{eq:extr-alt-p} when $p$ is constant and tends to zero. Thus,  investigating this general problem may indicate a new way of attacking the original conjecture. In particular, \eqref{eq:extr-alt-p} is confirmed to be true for $n\le 5$, as expressed by Corollary~\ref{cor:small-n}, and it seems legitimate to conjecture that $\max\{\F_p (f)\,:\,f\in\cb_n\}=2/e$ for all $n$, with the same expected extremal function $f_n$ as for the normalized Krzy\.z problem.
\par
Section~\ref{sect:extra-assumpt} concerns extremal functions for the special functionals $\F_p$ whose initial coefficients are assumed to vanish. There are several statements of this type in the literature. For the Krzy\.z problem, Brown \cite{Br1} showed that if $a_j=0$ for $1\le j <(n +1)/2$ for an extremal function, then the function must be $f_n$. He actually proved a more general version for $H^p$ spaces posed in \cite{HSZ}. In the Krzy\.z setting, a general theorem of this type involving the so-called inductive sets was later given in \cite{MSUV}. Recently, MacDonald \cite{MD} improved this by proving that if an extremal function $f$ for the Krzy\.z problem it suffices to assume only $a_1=\ldots=a_{\lfloor n/3\rfloor}=0$ to conclude that $f=f_n$ (see Corollary~\ref{cor:MacDonald} below). Here $\lfloor x\rfloor = \max\{k\in\Z:k\le x\}$ denotes the integer part of a real number $x$. Note that \cite{MD} uses the notation $l=\lceil \frac{n-2}{3} \rceil$, which gives the same integer as $\lfloor n/3\rfloor$.
	\begin{thm}\label{thm:Gn}
		Let $n\ge 6$, let $p$ be an admissible weight or $p\equiv 0$, and let $f(z)=\sum_{j=0}^\infty a_j z^j$ be a fixed extremal function for \eqref{eq:extr-alt-p} (for $p\equiv 0$, this means that $f$ is a normalized extremal function for the Krzy\.z problem \eqref{eq:extr-ref}). Then the following statements are equivalent:
		\begin{enumerate}[(a)]
			\item \label{thm:Gna} $a_1 = a_2 = \ldots = a_l = 0$ for $l = \lfloor n/3 \rfloor$.
			\item \label{thm:Gnb} $f(z) = e^{\frac{z^n-1}{z^n+1}}$. \end{enumerate}
If any of the above conditions is satisfied, then $G_n(p) = \frac{2}{e}$.
	\end{thm}
Corollary~\ref{cor:MacDonald} also covers the case $p\equiv 0$ of Theorem~\ref{thm:Gn}, thus giving a new proof of MacDonald's theorem.
Our proof differs a great deal from MacDonald's. It is based on the determinant estimates of Efraimidis \cite{Ef} (generalizing some classical results by Livingston \cite{L}).
\par\medskip
\section{Structural results for general non-linear extremal problems}
 \label{sect:gen-probl-struct}

\subsection{Existence and structure of extremal functions}
 \label{subsect:struct}
Recall that, by a theorem of Fatou, every bounded analytic function in $\D$ has radial limits $\widetilde{f} (e^{i\t})=\lim_{r\to 1^-} f(r e^{i\t})$ for almost every $\t\in [0,2\pi)$. A function bounded by one and such that $|\widetilde{f} (e^{i\t})|=1$ almost everywhere is said to be \textit{inner\/}. For a positive integer $N$, a \textit{singular inner function with $N$ atoms\/} is a function of the form
\begin{equation}
 f(z)= \beta \,\exp\(\sum_{j=1}^N r_j \frac{\a_j z + 1}{\a_j z - 1}\), \quad \quad r_j> 0, \quad |\a_j|=|\b|=1, \ 1\le j\le N.
 \label{eq:str-extr-fcn}
\end{equation}
where the points $\a_1,\ldots,\a_N$ are distinct. It is known that every extremal function for the Krzy\.z problem \eqref{eq:extr-ref} must have this form with $N\le n$; see Theorem~\ref{thm:A}~\ref{thm:Ai}. In spite of this information, proving that the only normalized extremal function is the function $f_n$ defined in Section~\ref{sect:intro} has remained elusive up to the time of writing this paper.

In this section we show that all extremal functions for the general problem \eqref{eq:extr-alt} must have a structure very similar to \eqref{eq:str-extr-fcn}, with one difference: constant functions of modulus one can also be extremal in some cases. Our proof follows an idea for the original Krzy\.z problem suggested by Donald E. Marshall \cite{Ma}, which in this case requires a careful further discussion in view of the generality of the problem. To this end, the following auxiliary statement will be very useful.

    \begin{lem} \label{lem:OpenMapping}
	Let $n \geq 1$ and let $h$ be a non-constant holomorphic function on $\D^{n+1}$. Then there exists $F \in \cb_\ast$ such that
		\[
			\re \Phi (F) > \re h(0, \ldots, 0).
		\]
    \end{lem}
    \begin{proof}
        This statement can be deduced in more than one way. Here is one simple proof.
\par
The restriction of $h$ defined by $g(z)=h(z,0,\ldots,0)$, $z\in\D$, is clearly analytic in $\D$. By the mean value property of the harmonic function $\re g$,
    \[
 \re g(0) = \frac{1}{2\pi}\int_0^{2\pi} \re g\(\tfrac25 e^{i\t}\)\,d\t ,
    \]
hence there exists $c_0$ with $|c_0|=\frac25$ (in particular, $\frac13<|c_0|<\frac12$) such that $\re g(c_0) \ge \re g(0)$.
\par
Since $h$ is not constant in $\D^{n+1}$, it cannot be constant on the open set
    \[
 \Omega = \left\{ (a_0,a_1,\ldots,a_n)\,:\, \frac13<|a_0|<\frac12, \, |a_1|, \ldots , |a_n| <  \frac1{3n} \right\} \subset \D^{n+1}
    \]
by the identity principle for holomorphic functions of several complex variables. Now for the value $c_0$ such that $\frac13<|c_0|<\frac12$ and $\re g(c_0) \ge \re g(0)$ whose existence we have established, choose a connected open neighborhood $U\subset \O$ of the point $(c_0, 0,\ldots, 0) \in \O$. Since $h$ is not constant on $U$, the open mapping theorem shows that $h(U)$ is an open set, hence there exists a point $(a_0, a_1,\ldots, a_n)\in U$ such that
    \[
 \re h (a_0, a_1,\ldots, a_n) > \re h(c_0,0,\ldots,0),
    \]
hence
    \[
 \re h (a_0, a_1,\ldots, a_n) > \re h(0,0,\ldots,0).
    \]

It only remains to check that for any given $(a_0, a_1,\ldots, a_n) \in \O$ there exists a function $F$ in $\cb_\ast$ such that $\F (F) = h(a_0, a_1,\ldots, a_n)$. This is easily seen by considering the polynomial
    \[
 F(z) = a_0 + a_1 z + \ldots + a_n z^n.
    \]
By our choice of $\O$, we have $F\in \cb_\ast$ since for all $z\in\D$ we have $|F(z)|\le |a_0|+\sum_{j=1}^n|a_j|<\frac12+\frac13<1$ and $|F(z)|\ge |a_0|-\sum_{j=1}^n|a_j| > \frac13-\frac13=0$.
    \end{proof}

We are now ready to prove the existence of extremal functions for \eqref{eq:extr-alt} and describe their structure. Note that there are some natural additional complications which do not appear in the study of more special problems such as the basic Krzy\.z problem.
\par
    \begin{thm} \label{thm:structure}
Let $n \geq 1$ and let $h: \overline{\D}^{n+1} \to \C$ be a non-constant function, holomorphic in $\D^{n+1}$ and continuous in $\overline{\D}^{n+1}$. Then there exists an extremal function for \eqref{eq:extr-alt}. Every such function $f$ is either a constant function of modulus one or a singular inner function as in  \eqref{eq:str-extr-fcn}, where $1\le N\le n$.

Moreover, if $h$ does not depend on its first variable $a_0$, then all extremal functions must be of the form \eqref{eq:str-extr-fcn}.
    \end{thm}
\begin{proof}
\textbf{Existence}. Denote by $a_k(f)$ the $k$-th Taylor coefficient of a function $f$ analytic in $\D$. In order to show that an extremal function for \eqref{eq:extr-alt} exists, let $(g_j)_j$ be a sequence of functions in $\cb_\ast$ such that $F_n = \lim_{j\to\infty} \re \Phi (g_j)$ (note that $F_n<\infty$ since $h$ is bounded on the compact set $\overline{\D}^{n+1}$). An application of Montel's theorem yields a subsequence, say $(f_j)_j$, that converges locally uniformly to some function $f$. By Weierstrass' theorem, $f$ is analytic in $\D$ and $\lim_{j\to\infty}  a_k(f_j) = a_k(f)$ for all $k\ge 0$. In view of the observations made in Section~\ref{sect:intro}, we have $|a_k(f_j)|\le 1$ for all $k\ge 0$ and all $j$, hence $|a_k(f)|\le 1$ for all $k\ge 0$. Moreover, since $0<|f_j|\le 1$ in $\D$ for all $j$, by Hurwitz's theorem either $f$ is identically zero or $f\in\cb_\ast$ (in which case it may still be a constant function). The assumption on continuity of $h$ in $\overline{\D}^{n+1}$ yields
\[
 \re h(a_0(f),a_1(f),\ldots,a_n(f)) = \lim_{j\to\infty} \re h(a_0(f_j),a_1(f_j),\ldots,a_n(f_j)) = F_n.
\]
Thus, $f$ will be extremal for \eqref{eq:extr-alt} if it is not identically zero. It is clear from  Lemma~\ref{lem:OpenMapping} that at least for the function $F$ found there $\re \F (F)$ yields a strictly greater value than the function identically equal to zero, hence the zero function cannot be extremal, and we are done.
\par\smallskip
\textbf{Structure}. First, we show that if $f$ is constant: $f\equiv a_0$, then $|a_0|=1$. By Lemma~\ref{lem:OpenMapping}, for the function $F$ obtained there we must have
\[
\re h(a_0,0,\ldots,0) = \re \Phi (f) \ge \re \Phi (F) > \re h(0, \ldots, 0).
\]
In particular, the function $g(z)=h(z,0,\ldots,0)$, holomorphic in $\D$ and continuous in $\overline{\D}$, is not constant. Since every constant function $w$ with $0<|w|\le 1$ belongs to $\cb_\ast$, we have $\re g(w)\le F_n=\re g(a_0)$ for $0<|w|\le 1$, and also for $w=0$ by continuity. Thus, $\re g$ attains its maximum over $\overline{\D}$ at the point $a_0$, and the maximum principle yields $|a_0|=1$, as claimed.

We now prove \eqref{eq:str-extr-fcn} in the case when $f$ is not constant. It suffices to prove the statement in the special case when $f(0)>0$. Otherwise we may consider
	\[
 f_0=\frac{\overline{a_0}}{|a_0|} f, \qquad
 h_0 (w_0,\ldots,w_n) =
 h\(\frac{a_0}{|a_0|} w_0,\ldots,\frac{a_0}{|a_0|}w_n\).
	\]
Note that multiplication by $\frac{\overline{a_0}}{|a_0|}$ is a bijective mapping of $\cb_\ast$ onto itself, hence $f_0 \in \cb_\ast$ if and only if $f\in \cb_\ast$. Also, $f_0(0)=|a_0|>0$ and $f_0$ is extremal for the functional corresponding to the function $h_0$ if and only if $f$ is extremal for the functional corresponding to the function $h$. Thus, it suffices to prove that $f_0$ (as an extremal function for $h_0$) has the structure \eqref{eq:str-extr-fcn} claimed, and it will immediately follow that $f$ also has the same structure.
\par
Since $f$ is not constant, we have $0<f(0)<1$, so we may write $f=e^g$, where $g(z)= \sum_{j=0}^\infty b_j z^j$ is analytic in $\D$, $\re g<0$ in $\D$, and $b_0=\log f(0)$ is real and negative. In what follows, we will consider a fixed value $b_0$ and hence also $a_0=e^{b_0}$ will be fixed. As we know from \cite{LS} or \cite{MSUV}, every coefficient of $f$ can be written as
$$
 a_k = a_0 P_k (b_1,\ldots,b_k),
$$
where $P_k$ is a polynomial (see also Lemma~\ref{lem:coeff-formula} below), hence
$$
 \F (f) = h(a_0,a_1,\ldots,a_n) = h(a_0,a_0 P_1(b_1), \ldots, a_0 P_n(b_1, \ldots, b_n))
$$
is again a holomorphic function of the variables $b_1$,\ldots,$b_n$, which can be written as another holomorphic function
$$
 \F (f) = H(\z_1,\ldots,\z_n)
$$
of the variables $\z_1=-b_1$, \ldots, $\z_n=-b_n$. Note that the function
$$
 -g(z) = -b_0 -b_1 z - \ldots - b_n z^n - \ldots
$$
is analytic in $\D$ and has positive real part there, so we can apply to it the theorems of Carath\'eodory and Toeplitz which can be found in \cite[Chapter~IV]{Ts}, among other sources.
\par
The set
$$
 \rad_n = \left\{ (w_1, \ldots, w_n)\in\C^n\,:\,
 \begin{vmatrix}
 -2 b_0 & w_1
\\
 \overline{w}_{1} & -2 b_0
\end{vmatrix} >0
, \ldots ,
 \begin{vmatrix}
 -2 b_0 & w_1 & \cdots & w_n
\\
 \overline{w}_1 & -2 b_0 & \cdots & w_{n-1}
\\
 \vdots & \vdots & \ddots & \vdots
\\
 \overline{w}_n & \overline{w}_{n-1} & \cdots & -2 b_0
\end{vmatrix}
  >0 \right\}
$$
is open, being a finite intersection of open sets. By Sylvester's criterion, $\rad_n$ is the set of vectors for which the corresponding Hermitian Toeplitz matrix is positive definite; hence $\rad_n$ is also convex, and it contains the zero vector $(w_1, \ldots, w_n)=(0, \ldots, 0)$. By the Carath\'eodory--Toeplitz theorem, the quadratic form
associated to the Hermitian Toeplitz matrix corresponding to the coefficient vector $(\z_1, \ldots, \z_n)$ of $-g$ is non-negative. Therefore, for $0<r<1$, the vectors $(r\z_1, \ldots, r\z_n)$ belong to $\rad_n$, since their Toeplitz matrices are obtained by adding the positive definite matrix $-2(1-r)b_0 I$ to $r$ times the former one. In particular, $(\z_1, \ldots, \z_n)$ belongs to the closure of $\rad_n$. If $H$ were constant in $\rad_n$, continuity would give $H(\z_1, \ldots, \z_n)=H(0,\ldots,0)$. Thus, the constant function $f\equiv a_0=e^{b_0}<1$ would also be extremal, which is impossible by the first part of the proof. Hence, $H$ is not constant in $\rad_n$.
\par
Suppose now, to reach a contradiction, that $(\z_1, \ldots, \z_n)\in \rad_n$. Since $\rad_n$ is open and connected and $H$ is not constant there, the open mapping theorem for holomorphic functions of several complex variables implies the existence of $(w_1, \ldots, w_n)\in \rad_n$ such that
$$
 \re H(\z_1, \ldots, \z_n) < \re H(w_1, \ldots, w_n).
$$
On the other hand, by a theorem of Carath\'eodory \cite[Theorem~IV.24]{Ts}, there exists a function $-g^\ast$, analytic in $\D$ and with positive real part, such that
$$
 -g^\ast(z) = -b_0 + w_1 z + \ldots + w_n z^n + \ldots
$$
The function $f^\ast= e^{g^\ast}$ belongs to $\cb_\ast$ and satisfies $f^\ast(0)=a_0$ and $\F(f^\ast)=H(w_1,\ldots,w_n)$, hence
$\re \F (f) < \re \F (f^\ast)$, which contradicts the assumption that $f$ is extremal.
\par
Therefore, $(\z_1,\ldots,\z_n)\notin\rad_n$. Since all the Toeplitz determinants associated with $-g$ are nonnegative and the first one equals $-2b_0>0$, there is a smallest integer $N$ with $1\le N\le n$ such that
$$
 \begin{vmatrix}
 -2 b_0 & -b_1 & \cdots & -b_N
\\
 -\overline{b}_1 & -2 b_0 & \cdots & -b_{N-1}
\\
 \vdots & \vdots & \ddots & \vdots
\\
 -\overline{b}_N & -\overline{b}_{N-1} & \cdots & -2 b_0
\end{vmatrix}
 = 0.
$$
The Carath\'eodory--Toeplitz theorem \cite[Theorem~IV.22]{Ts} implies that, for some $r_j>0$ and distinct points $\a_j$ with $|\a_j|=1$,
$$
 - g(z) = \sum_{j=1}^N r_j \frac{1+\a_j z}{1-\a_j z},
$$
which proves the structural formula \eqref{eq:str-extr-fcn} for $f$.
\par
Finally, suppose that $h$ does not depend on its first variable. If $f\equiv a_0$ is a constant function, then $a_1=\ldots=a_n=0$, hence $\Phi(f)=h(a_0,0,\ldots,0)=h(0,0,\ldots,0)$. Thus, for the function $F$ given by Lemma~\ref{lem:OpenMapping} we have
    \[
 F_n \ge \re \Phi (F) > \re h(0,\ldots,0) = \re \Phi (f),
    \]
hence no constant function can be extremal.
\end{proof}
\par
\subsection{Uniqueness and normalizations}
 \label{subsect:unique-normal}
It is important to notice that the structure of extremal functions is not quite the same in the most  general case as in the Krzy\.{z} problem. Also, for the sake of trying to obtain uniqueness, in some cases it may be convenient to consider our extremal problems over a smaller class instead of over the entire class $\cb_\ast$. For example, for the original Krzy\.{z} problem we can use rotations as above to see that the maximum $M_n$ defined in \eqref{eq:extr-ref} is actually equal to
    \[
 M_n = \sup \left\{ \re a_n\,\colon\, f\in\cb_\ast, \,  f(0)=a_0>0\right\}.
    \]
We will refer to this formulation as the \textit{normalized Krzy\.z problem\/}. Three basic questions arise here:
\par
(1) Can a constant function of modulus one actually be extremal for some choice of $\F$ in the general problem \eqref{eq:extr-alt}? A trivial example would be $\F (f)=a_0$, but it is desirable to have an example where $\F$ does not depend only on $a_0$.
\par
(2) How many extremal functions can we have for a given problem? Is it possible to have only one such function, or maybe some other small finite number?
\par
(3) Can we impose further normalizations as in the Krzy\.z problem, such as $f(0)>0$? How does this affect the extremal functions, and do we get uniqueness?
\par
In what follows, we answer all these questions by providing appropriate examples.
\par\smallskip
Our first example answers our question (1) raised above and, regarding question (2), shows that an extremal problem can have a unique solution.
\par
    \begin{ex} \label{ex:const-extremal}
The extremal problem associated to the linear functional $\Phi(f) = 2a_0 + a_n$ is maximized only by the constant function 1. Recalling  that for $f\in\cb_\ast$ we have $|a_n|\le 1 - |a_0|^2$ (see Section~\ref{sect:intro}), this can be seen as follows:
\[
 \re (2a_0 + a_n) \le 2 \re a_0 + (1 - |a_0|^2) \le 1 + 2 \re a_0 - (\re a_0)^2 \le 2,
\]
and the maximum value $2$ is only attained when $\re a_0 = 1$, which immediately forces $f$ to be constant: $f\equiv 1$. Thus, in this case there is only one extremal function and it is constant.
    \end{ex}
\par
Note that the functionals $\Phi (f)$ which are only functions of $a_0$, that is,
\[
 h(a_0,a_1,\ldots,a_n) = h(a_0,0,\ldots,0), \quad \forall (a_0,a_1,\ldots,a_n) \in \D^{n+1},
\]
give rise to some unusual situations. Here is an example.

    \begin{ex} \label{ex:2-extremals}
Let $\Phi (f)=a_0-a_0^2$.
		\begin{enumerate}[(a)]
			\item \label{ex:2-extremalsa}
Let us first maximize $\re \Phi (f)$ over $\cb_\ast$. The value  $a_0=f(0)$ must satisfy $0<|a_0|\le 1$. Writing $a_0=x+yi$ with $x$, $y\in\R$ and $0<x^2+y^2\le 1$, we see that
\[
 \re \Phi (f) = x - x^2 + y^2 \le 1 + x - 2 x^2 \le \frac98.
\]
The maximum is attained only when equality holds throughout; that is, only for
\[
 x=\frac14, \quad y = \pm \frac{\sqrt{15}}{4},
\]
and in this case $a_0$ clearly has modulus $1$, forcing $f\equiv a_0$. Thus, there are exactly two extremal functions and they are both constant.
\item \label{ex:2-extremalsb}
If we want to maximize $\re \Phi (f)$ over the reduced class of all $f\in\cb_\ast$ with $a_0>0$, we obtain
\[
  \re \Phi (f) = a_0-a_0^2 = a_0 (1-a_0) \le \frac14,
\]
with equality only for $a_0=\frac12$. However, in this case it is immediate that there are many extremal functions different from those described by Theorem~\ref{thm:structure}, e.g., every function of the form
\[
 f_\la (z) = \frac12\,(1+\la z), \quad |\la|\le 1.
\]
All these functions are obviously bounded by one, non-vanishing in $\D$, and have the property $f(0)=\frac12$. Even the constant function among them ($\la=0$) is clearly different from the ones from Theorem~\ref{thm:structure} and the ones found in the previous analysis.
\end{enumerate}
    \end{ex}
\par
Both Example~\ref{ex:const-extremal} and Example~\ref{ex:2-extremals}~\ref{ex:2-extremalsa} are in sharp contrast with the Krzy\.z problem, for which there are infinitely many extremal functions if we do not assume any normalization. Indeed, at least one extremal function $f$ for the Krzy\.z problem exists, and for each $\la$ with $|\la|=1$ the function $\overline{\la}^n f(\la z)$ has the same $n$-th coefficient as $f$; since its value at the origin is $\overline{\la}^n a_0\neq 0$, there are uncountably many distinct extremal functions.
\par
As our Example~\ref{ex:2-extremals}~\ref{ex:2-extremalsb} shows, the structure obtained in Theorem~\ref{thm:structure} can be completely ruined under the  additional assumption $f(0)>0$. One way of getting around this obstacle in order to be able to assume that there is an extremal function with a positive value at the origin is as follows. Define two functions $h$ and $\widetilde{h}$ that satisfy the conditions of Theorem~\ref{thm:structure} to be \textit{associated\/} if there exists $\la$ with $|\la|=1$ such that $\widetilde{h}(z_0,z_1,\ldots,z_n) = h(\la z_0,\la z_1,\ldots,\la z_n)$.
This clearly defines an equivalence relation, so instead of individual functionals we can consider equivalence classes of associated functionals. In each class we can choose one representative, $h$, such that there exists an extremal function $f$ for the functional $\Phi (f) = h(a_0,a_1,\ldots,a_n)$ with $a_0>0$, because we can find $\la$ with $|\la|=1$ such that $\la a_0 > 0$. This may provide a satisfactory solution (albeit not an ideal one).
\par
The next example shows that the problem encountered in our Example~\ref{ex:2-extremals} does not come from the fact that the functional depended only on $a_0$. There are other functionals that give rise to the same situation, even though the extremal function fits within the description given in Theorem~\ref{thm:structure}.
    \begin{ex} \label{ex:no-extremal-posit}
Let $n=1$ and $\Phi (f)=a_1-a_0^2$. We already know that there is an extremal function for this functional but there are none with the property $f(0)>0$. For a general $f\in\cb_\ast$ we have
    \[
 \re \F (f) = \re a_1 - \re a_0^2 \le 1 - |a_0|^2 - \re a_0^2 \le 1.
    \]
Equality in the Schwarz--Pick lemma for a non-constant function is only possible for a disc automorphism, which has a zero in $\D$, so this case must be ruled out. Thus, equality in the above chain can only hold for a constant function $f\equiv a_0$ with $|a_0|=1$ and $\re a_0=0$, that is, $f\equiv \pm i$. Hence, there are only two extremal functions, both constant functions of modulus one. However, under the assumption $a_0>0$ we get
    \[
 \re \F (f) = \re a_1 - a_0^2 \le \re a_1 \le \frac{2}{e},
    \]
as is well known. Thus, no extremal function has the property $f(0)>0$.
\par
In order to obtain an associated functional with an extremal function that has a positive value at the origin, instead of $\F$ we would have to consider the functional $\Phi_i (f) = i a_1 - (i a_0)^2 = a_0^2 + i a_1$, whose extremal functions are $f\equiv \pm 1$.
    \end{ex}
\par
\subsection{Further properties of extremal functions}
 \label{subsect:more-prop}
In what follows, we will refer to any function of the form \eqref{eq:str-extr-fcn} as a \textit{candidate function\/}. We will use the notation $\cc_N$ for the set of all candidate functions with exactly $N$ atoms, $1\le N \le n$, while $\cc_0$ will denote the class of constant functions of modulus one (no atoms in the representation in Theorem~\ref{thm:structure}). Thus, every extremal function for \eqref{eq:extr-alt} belongs to the set $\bigcup_{N=0}^n \cc_N$. The following fact was observed in \cite[Remark~3.2]{AM} for the Krzy\.z problem, and this easily generalizes to our situation. It shows that finding the maximum of \,$\re \Phi (f)$ over all admissible functions reduces to finding the supremum of the same quantity only over the functions with the full number of atoms.
    \begin{rem} \label{rem:est-atoms}
Let $n \geq 1$, $0\le N<n$, and let $h: \overline{\D}^{n+1} \to \C$ be a non-constant function, holomorphic in $\D^{n+1}$ and continuous in $\overline{\D}^{n+1}$. Then
$$
 \sup\,\{ \re \Phi (f)\,:\,f\in\cc_N\} \le  \sup\,\{ \re \Phi (f)\,:\,f\in\cc_{N+1}\}.
$$
In particular,
$$
 \sup\,\left\{ \re \Phi (f)\,:\,f\in \bigcup_{N=0}^{n-1}\cc_N\right\} \le \sup\,\{ \re \Phi (f)\,:\,f\in\cc_n\}.
$$
Hence,
$$
 F_n = \max \{ \re \Phi (f) \,\colon\,f\in\cb_\ast \} = \sup\,\{ \re \Phi (f)\,:\,f\in\cc_n\}.
$$
To see this, it suffices to show that if $\re \Phi (f) \le M$ for every candidate function with $N+1$ atoms, $0\le N<n$, then the same is true for every candidate function with $N$ atoms. Given $f\in\cc_N$, if we  choose $\a$ with $|\a|=1$ different from the atoms of $f$ (if $N=0$, any $\a$ will do). Then the function
$$
 g_r(z) = f(z)\, e^{r \frac{\a z + 1}{\a z - 1}}, \qquad r>0,
$$
belongs to $\cc_{N+1}$. It is elementary to see that $r \frac{\a z + 1}{\a z - 1}\to 0$ uniformly on compact sets, hence $e^{r \frac{\a z + 1}{\a z - 1}}\to 1$, also  uniformly on compact sets. By Weierstrass' theorem, it is immediate that the Taylor coefficients of $g_r$ converge to the Taylor coefficients of $f$ and therefore  also $\re \Phi (g_r) \to \re \Phi (f)$ as $r\to 0^+$. Thus, any uniform bound for $\re \Phi (f)$ over the class $\cc_{N+1}$ is also a uniform bound for the functions in $\cc_N$.
    \end{rem}
It has recently been proved by Lei and Zhang \cite{LZ} (indirectly, by proving the Agler--McCarthy conjecture \cite{AM}) that if every extremal function for the original Krzy\.z problem has exactly $n$ atoms, then the Krzy\.z conjecture is true. Note, however, that even for the Krzy\.z problem it has not yet been established that \textit{every\/} extremal function must have the full number of atoms: $N=n$, nor that there exists an extremal function with this number of atoms.

It is well-known to the experts that if there is a unique extremal functions for the Krzy\.z conjecture, then the conjecture holds. This fact admits a partial extension to our generalized problem by following  the same argument.
\par
\begin{thm} \label{thm:Uniqueness}
		Let $h : \overline{\D}^2 \to \C$ be a non-constant function holomorphic in $\D^2$ and continuous in $\overline{\D}^2$, and consider the functional $\Phi (f) = h(a_0, a_n)$. If the associated extremal problem \eqref{eq:extr-alt} has exactly one extremal function $f$ with $a_0>0$, then $f(z)=g(z^n)$ for some $g\in \cb_\ast$.
	\end{thm}
    \begin{proof}
		If $\omega$ is a primitive $n$-th root of unity, then $f_j (z) = f(\omega^j z)$ has the same coefficients $a_0$ and $a_n$ as $f$, hence it is also extremal for all $0 \leq j \leq n-1$. By uniqueness, $f = f_j$ for all $j$ and, in particular, $f = W_n f$ where $W_n f$ is the F.~Wiener transform of $f$ defined by \eqref{eq:wiener} in Section~\ref{sect:intro}. Thus, $f(z)=W_n f(z)=g(z^n)$, where $g$ is analytic in $\D$; since $f\in\cb_\ast$, also $g\in\cb_\ast$.
    \end{proof}
Note that, in spite of the restrictive condition on the coefficients, the above result covers the original Krzy\.z problem. Moreover, the same argument applies verbatim to extremal problems posed over any subclass of $\cb_\ast$ that is invariant under the rotations $f\mapsto f(\om z)$, such as the class $\cb_n$ considered in the next section.

\par\medskip
\section{A special class of problems}
 \label{sect:spec-probl}
\par
Although our Theorem~\ref{thm:structure} is valid for quite general functionals, we will now focus on a very specific subclass of functionals which will be useful for our purpose and to which Theorem~\ref{thm:Uniqueness} also applies. As a motivation, we first notice the following.
\par
By Theorem~\ref{thm:A}~\ref{thm:Ai} and \ref{thm:Aii}, every normalized extremal function $f$ for the Krzy\.z problem \eqref{eq:extr-ref} satisfies $a_n\ge 2 a_0>0$ and belongs to one of the classes $\cc_N$ of candidate functions, with $1\le N\le n$. Thus, we actually have
    \[
 M_n = \max \{ \re a_n \,:\, f\in \cb_\ast\} = \max \{ a_n \,:\, f\in \cb_n\},
    \]
where
\begin{equation}
 \cb_n = \left\{f\in \cb_\ast\,:\,f(z) = \sum_{j=0}^\infty a_j z^j \text{ in } \D,\, \ 0<2 a_0\le a_n, \ f\in \bigcup_{N=1}^n \cc_N\right\}.
 \label{eq:restr-class}
\end{equation}
Clearly, this class encodes all the information needed about possible extremal functions for the Krzy\.z problem. We note that $\cb_n$ contains an uncountable collection of functions with exactly $n$ atoms. Namely, let $r>0$ and define
    \[
 g_r (z) = - r \frac{1 - z^n}{1 + z^n} = - r (1 - 2 z^n + 2 z^{2n} - 2 z^{3n} + \ldots ), \qquad   f_r = e^{g_r}.
    \]
Then $f_r$ is a singular inner function with exactly $n$ atoms and
    \[
 a_0 = e^{- r} > 0, \quad a_n = 2 r e^{-r} > 0.
    \]
Thus, we have $a_n \ge 2 a_0$ whenever $r\ge 1$. Note that for $r=1$ we obtain the conjectured extremal function $f_n$.

One can also find examples of functions in $\cb_n$ with fewer than $n$ atoms when $n\ge 2$.

    \begin{ex} \label{ex:more-atoms}
Let $h_r(z) = \exp \(r \frac{z-1}{z+1}\)$, $r>0$, which is a singular inner function with one atom. We can express this function as
    \[
 h_r(z) = e^{-r} \sum_{j=0}^\infty (-1)^j L_j^{(-1)} (2r) z^j,
    \]
where $L^{(-1)}_j$ are the Laguerre polynomials of order $-1$; \textit{cf.\/} \cite[p.~258]{LS}. Thus, $h_r\in\cb_n$ if and only if $(-1)^n L^{(-1)}_n (2r) \ge 2$. This polynomial can be explicitly computed \cite[p.~775]{AS} as
\[
 L^{(-1)}_n (2r) = \sum_{j=0}^n (-1)^j \binom{n-1}{n-j} \frac{(2r)^j}{j!}.
\]
Since the leading coefficient of the polynomial $(-1)^n L^{(-1)}_n (2r)$ in $r$ is positive, this holds for all sufficiently large $r$.
\end{ex}
\smallskip

In the rest of this section we consider special functionals of the form \eqref{eq:spec-funct}, which are convex combinations of $2 a_0$ and $a_n$, as defined in Section~\ref{sect:intro}. We refer to such functionals as \textit{admissible functionals\/} and to the corresponding functions $p$ as \textit{admissible weights\/}. For $n\ge 1$ and for each admissible weight $p$, we consider the extremal problem of finding the supremum defined in \eqref{eq:extr-alt-p}.

Here are a few examples of such functionals:
    \[
 \F (f) = 2 a_0, \qquad \F (f) = 2 a_0^2 + (1-a_0) a_n, \qquad \F (f) = 2 a_0 a_n + (1-a_n) a_n
    \]
which correspond, respectively, to the weights
    \[
 p(x,y)=1, \qquad p(x,y)=x, \qquad p(x,y)=y.
    \]
Since $\F_p(f)=a_n-p(a_0,a_n)\,(a_n-2a_0)$, for all $f\in\cb_n$ we have
	\begin{equation}
  2 a_0 \le \F_p (f) \le a_n, \qquad \F_p (f) = a_n \Longleftrightarrow a_n = 2 a_0 .
    \label{eq:key-ineq}
	\end{equation}
By Theorem~\ref{thm:A}~\ref{thm:Aiii}, if a fixed normalized extremal function $f$ for the Krzy\.z problem \eqref{eq:extr-ref} satisfies $a_n=2a_0$, then $M_n=2/e$ and $f=f_n$. This is precisely what makes the functionals $\F_p$ interesting for our study, in an attempt to ``tighten'' as much as possible the inequality $2 a_0\le a_n$ and strive for equality.

The classical Krzy\.z problem corresponds to the weight $p\equiv 0$, which is not admissible, since the equivalence in \eqref{eq:key-ineq} fails for it. Nevertheless, it will be convenient to allow this limit case in some statements, and we will always say so explicitly. Note that for constant weights $p\equiv c$ with $c\in (0,1]$ we easily obtain
    \[
 (1 - c) M_n \le G_n(c) \le M_n,
    \]
so the classical extremal value $M_n$ remains the limit of these problems as $c\to 0^+$. Note, however, that this convergence does not itself establish equality for any $c\in (0,1]$.
\smallskip

Recall that Example~\ref{ex:const-extremal}, which concerns the functional $2a_0+a_n=2\F_{1/2}(f)$ corresponding to the weight $p\equiv 1/2$, shows that constant extremal functions are possible when the maximum is taken over the whole class $\cb_\ast$. Here we avoid this problem by imposing the structure from Theorem \ref{thm:structure} in view of our definition \eqref{eq:restr-class} of the restricted class $\cb_n$.
	\begin{prop} \label{prop:existence}
		Let $\F_p$ be as in \eqref{eq:spec-funct}. Then $G_n\ge 2/e$ and the supremum in \eqref{eq:extr-alt-p} is attained, that is, there exists an extremal function $f\in\cb_n$ (by definition, $f$ is a singular inner function with at most $n$ atoms).
	\end{prop}
	\begin{proof}
	The function $f_n(z)= e^{\frac{z^n-1}{z^n+1}} = \frac{1}{e} + \frac{2}{e} z^n + \ldots$ belongs to $\cb_n$ and $\F_p(f_n)=2/e$, hence $G_n\ge 2/e$. Let $(g_j)_j$ be a sequence in $\cb_n$ such that $\F_p(g_j)\to G_n$. Since $g_j(0)>0$, each $g_j$ can be written as
\[
 g_j(z)=\exp\(-\sum_{i=1}^{N_j} r_{ij}\,\frac{1+\a_{ij}z}{1-\a_{ij}z}\), \qquad N_j\le n,\quad r_{ij}>0,\quad |\a_{ij}|=1.
\]
Let $s_j=\sum_{i=1}^{N_j} r_{ij}=-\log g_j(0)$. If $s_j\to\infty$ along a subsequence, then the estimate $|g_j(z)|\le \exp\(-s_j\frac{1-|z|}{1+|z|}\)$ shows that $g_j\to 0$ locally uniformly along that subsequence, so that $a_0(g_j)\to 0$, $a_n(g_j)\to 0$ and, by the continuity of $p$ at $(0,0)$, $\F_p(g_j)\to 0$, which contradicts $G_n\ge 2/e$. Hence, $(s_j)_j$ is bounded and, passing to a subsequence, we may assume that $N_j=N$ for all $j$, $r_{ij}\to r_i\ge 0$ and $\a_{ij}\to\a_i$ for $1\le i\le N$. Then $g_j\to f$ locally uniformly in $\D$, where
\[
 f(z)=\exp\(-\sum_{i=1}^{N} r_{i}\,\frac{1+\a_{i}z}{1-\a_{i}z}\)
\]
is either a singular inner function with at most $n$ atoms (some atoms may merge or disappear in the limit) or the constant function $1$. Since the Taylor coefficients converge, we have $0<2a_0(f)\le a_n(f)$; in particular, $f$ is not constant. Thus, $f\in\cb_n$ and, by the continuity of $p$ in $\D^2$, $\F_p(f)=\lim_{j\to\infty}\F_p(g_j)=G_n$.
	\end{proof}
\par
By Theorem~\ref{thm:A}~\ref{thm:Aiii}, proving the equality $a_n=2a_0$ for a fixed normalized extremal function would prove the statement of the Krzy\.z conjecture for that specific function. It is also known that uniqueness of the normalized extremal function is equivalent to the full statement of the conjecture. Our Theorem~\ref{thm:rel-probl}, formulated in the Introduction, collects the analogous relations between the problems \eqref{eq:extr-ref} and \eqref{eq:extr-alt-p}. Its proof is simple, but it shows that the two problems are closely related and gives us the flexibility to change from one functional to another when convenient. In particular, part (b) shows that the apparently weaker assumption $\sup \left\{ 2 a_0\,\colon\,f\in\cb_n \right\} = \sup \left\{ a_n\,\colon\,f\in\cb_n \right\}$ already implies that $M_n=2/e$. Thus, we believe the result may prove useful for the further study of the Krzy\.z problem.
\par\smallskip
    \begin{proof}[Proof of Theorem~\ref{thm:rel-probl}]
(a) Let $f(z)=\sum_{j=0}^\infty a_j z^j$ be an extremal function for \eqref{eq:extr-alt-p}; it exists by Proposition~\ref{prop:existence}. Since $f_n\in\cb_n$ and $\F_p(f_n)=2/e$, the inequalities \eqref{eq:key-ineq} and the inclusion $\cb_n\subset\cb_\ast$ yield
$$
 \frac2e\le G_n(p)=\F_p(f)\le a_n\le M_n=\frac2e .
$$
Hence, $G_n(p)=2/e$ and $\F_p(f)=a_n=M_n$. Thus, $f$ is also a normalized extremal function for the Krzy\.z problem \eqref{eq:extr-ref} and, by \eqref{eq:key-ineq}, $a_n=2a_0$. By Theorem~\ref{thm:A}~\ref{thm:Aiii}, $f=f_n$.
\par\smallskip
(b) By Proposition~\ref{prop:existence}, there is an extremal function $f$ for \eqref{eq:extr-alt-p}. For any such function $f(z)= \sum_{j=0}^\infty a_j z^j$, the inequalities \eqref{eq:key-ineq} and the inclusion $\cb_n\subset\cb_\ast$ yield
$$
 M_n = G_n(p) = \F_p (f) \le a_n \le M_n .
$$
Hence, $\F_p(f)=a_n=M_n$, so $f$ is also a normalized extremal function for the Krzy\.z problem \eqref{eq:extr-ref} and, by \eqref{eq:key-ineq}, $a_n=2a_0$. By Theorem~\ref{thm:A}~\ref{thm:Aiii}, $M_n=2/e$ and $f=f_n$. Since this holds for every extremal function for \eqref{eq:extr-alt-p}, $f_n$ is the only one. The particular case follows by taking $p\equiv 1$, since $G_n(1)=\sup\{2a_0\,\colon\,f\in\cb_n\}$ and $M_n=\sup\{a_n\,\colon\,f\in\cb_n\}$.
\par\smallskip
(c) Let $\om=e^{(2\pi i)/n}$. The class $\cb_n$ is invariant under the rotations $f\mapsto f(\om^j z)$, which do not change the coefficients $a_0$ and $a_n$. Hence, proceeding as in the proof of Theorem~\ref{thm:Uniqueness}, we see that the unique extremal function satisfies $f(z)=g(z^n)$ for some $g\in\cb_\ast$ with $g(0)=a_0>0$ and $g'(0)=a_n$. By the classical case $n=1$ of the Krzy\.z conjecture, we have $a_n=g'(0)\le 2/e$, with equality if and only if $g(z)=e^{\frac{z-1}{z+1}}$. Arguing like in (a), we see that $G_n(p)=2/e$, $g(z)=e^{\frac{z-1}{z+1}}$, and $f(z)=g(z^n)=e^{\frac{z^n-1}{z^n+1}}$.
    \end{proof}
\par
Since the Krzy\.z conjecture has been proved for $n\le 5$ and a computer-assisted proof for $n=6$ is currently under scrutiny (see Section~\ref{sect:intro}), Theorem~\ref{thm:rel-probl}~\ref{thm:rel-probl:a} yields the following.
	\begin{cor} \label{cor:small-n}
	Let $1\le n\le 5$ and let $\F_p$ be an admissible functional. Then $G_n(p)=2/e$ and $f_n$ given by \eqref{eq:extr-fcn} is the only extremal function for \eqref{eq:extr-alt-p}.
	\end{cor}
\par\smallskip
We hope that Theorem~\ref{thm:rel-probl} may provide some new directions in the research on the Krzy\.z problem. The intermediate values of $p$ can be seen as a ``bridge'' between the original Krzy\.z conjecture and the problem of maximizing the constant term within the class $\cb_n$ and may also deserve a closer look.
    \begin{rem} \label{rem:coeff-comp}
	Comparing the extremal functions for the general problem and the original Krzy\.z problem allows us to deduce some conclusions about the constant term of an extremal function. For instance, if $G_n(p)=2/e$ for some admissible weight $p$, then every normalized extremal function $g(z) = \sum_{j=0}^{\infty} b_j z^j$ for the Krzy\.z problem \eqref{eq:extr-ref} satisfies $b_0 \leq 1/e$. Indeed, $g\in\cb_n$ by Theorem~\ref{thm:A}, and \eqref{eq:key-ineq} yields $2 b_0 \le \F_p (g) \le G_n (p) = 2/e$.
    \end{rem}
    \begin{rem} \label{rem:total-mass}
	The case $p\equiv 1$ admits a natural reformulation in terms of the singular measure associated with $f$. If $f\in\cb_n$ is written as
	\[
	 f(z)=\exp\(-\sum_{j=1}^N r_j\,\frac{1+\a_jz}{1-\a_jz}\),
	\]
	then $a_0=e^{-s}$, where $s=\sum_{j=1}^N r_j$ is the total mass of this measure. Hence, $G_n(1)=2e^{-s_\ast}$, where $s_\ast$ is the minimum of the total mass over all $f\in\cb_n$; it is attained by Proposition~\ref{prop:existence}. Since $f_n$ has total mass $1$, the conjecture formulated in Subsection~\ref{subsect:equiv-form} for $p\equiv 1$ states that \emph{every singular inner function $f$ with at most $n$ atoms, $f(0)>0$ and $a_n\ge 2a_0$ has total mass at least $1$, with equality only for $f=f_n$.} Moreover, by Theorem~\ref{thm:rel-probl}~\ref{thm:rel-probl:b}, the equality $2e^{-s_\ast}=M_n$ would imply that $M_n=2/e$.
    \end{rem}

    In the same spirit as above, we can also prove the following result. Observe that the bound in (c) is only of interest when $p>1-2/e$, since otherwise the maximum on the right-hand side is zero.
	\begin{prop} \label{prop:compare}
		Let $\F_p$ be an admissible functional and let $f(z)=\sum_{j=0}^{\infty} a_j z^j$ be an arbitrary but fixed extremal function for \eqref{eq:extr-alt-p}. Let $g(z)=\sum_{j=0}^{\infty} b_j z^j$ be an arbitrary but fixed normalized extremal function for the Krzy\.{z} problem \eqref{eq:extr-ref}. Then the following statements hold.

	\begin{enumerate}[(a)]
\item \label{prop:compare:a}
	If $p(a_0, a_n) \geq p(b_0, b_n)$ then $b_0 \leq a_0$.
\item \label{prop:compare:b}
    If $p$ is a positive constant, then $b_0 \leq a_0$.
\item \label{prop:compare:c}
    If $p$ is a constant with $0<p<1$, then
    \[
 a_0 \geq \max \left\{ \frac{p-\sqrt{p^2-(1-p)\(\frac2e-1+p\)}}{1-p},\,0\right\},
    \]
    and if $p\equiv 1$, then $a_0\ge 1/e$.
    \end{enumerate}
	\end{prop}
\begin{proof}
	For part (a), we see that
	\[
 p(a_0, a_n) (2a_0-a_n)+a_n \geq p(b_0,b_n)(2b_0-b_n)+b_n \geq p(a_0,a_n)(2b_0-b_n)+b_n.
    \]
    Consequently, it follows that
	\[
 2p(a_0,a_n)(a_0-b_0) \geq (1-p(a_0,a_n))(b_n-a_n) \geq 0,
    \]
    which proves that $a_0 \geq b_0$ if $p(a_0,a_n)>0$. If $p(a_0,a_n)=0$, then $p(b_0,b_n)=0$ by hypothesis, so the first chain of inequalities gives $a_n\ge b_n$; moreover, $a_n=2a_0$, since $p(x,y)>0$ whenever $0<2x<y<1$. Hence, $2b_0\le b_n\le a_n=2a_0$.
\par
	Part (b) follows automatically from part (a) since if $p$ is a constant, then $p(a_0, a_n) = p(b_0, b_n)$.
\par
	Finally, for part (c), assume that $p$ is a constant. Since $f_n\in\cb_n$, we have
	\[2a_0p+(1-p)a_n = G_n \geq \F_p(f_n) = 2/e.\]
	If $p=1$, this is $a_0\ge 1/e$. If $0<p<1$, we use the estimate $a_n\le 1-a_0^2$ from Section~\ref{sect:intro} to obtain $Q(a_0)\le 0$, where
	\[
	 Q(x)=(1-p)x^2-2px+\(\frac2e-1+p\).
	\]
	The discriminant of $Q$ is $4\(p^2-(1-p)\(\frac2e-1+p\)\)=4\(2p^2-2\(1-\frac1e\)p+1-\frac2e\)$, which is positive since $4\(1-\frac1e\)^2-8\(1-\frac2e\)=4\(\frac2e+\frac1{e^2}-1\)<0$. Hence, $a_0$ is not smaller than the smaller root of $Q$, which is the first expression in the maximum, and the claim follows since $a_0>0$.
\end{proof}
\begin{rem} \label{rem:compare}
Using the crude estimate $a_n<1$ instead of $a_n\le 1-a_0^2$ in the proof of part (c), one obtains the simpler bound $a_0\ge \frac1p\(\frac1e-\frac{1-p}2\)$, which is positive exactly when $p>1-2/e$.
\end{rem}

\par\medskip
\section{Problems with further assumptions on coefficients}
 \label{sect:extra-assumpt}
\par
We continue to study the same type of problems as in the previous section, this time under some additional conditions on the coefficients, in the vein of \cite{MSUV}. But first we need to review a number of technical tools that will be needed.
\subsection{Some estimates for the Carath\'eodory class}
 \label{subsect:carath}
Throughout this section, we will denote by $\cp$ the \textit{Carath\'eodory class\/} of normalized analytic functions $q$ in $\D$ with positive real part:
    \[
        q(z)=\sum_{j=0}^{\infty} q_j z^j, \quad q_0=q(0)=1.
    \]
It is customary to denote the functions in $\cp$ by $p$ and their Taylor coefficients by $p_k$ but here we use the letter $q$ to avoid confusion with the admissible weights $p$ of Section~\ref{sect:spec-probl}. It is well known that every function $q\in\cp$ has the form
\[
    q(z) = \int_0^{2\pi} \frac{1 + e^{- i t}z}{1 - e^{- i t} z}\, d\mu (t),
\]
where $\mu$ is a Borel probability measure on $[0,2\pi)$ (which can also be viewed as the unit circle with an appropriate identification). This follows directly from \cite[Section~1.9]{D} upon imposing the normalization condition $q(0)=1$. The above representation is known as the \textit{Herglotz representation\/} of $q$ and $\mu$ is referred to as its \textit{Herglotz measure\/}. Among the Herglotz measures, the ones of special interest to us are the ones supported at $N$ point masses on the unit circle $\{\a_j = e^{- i t_{j}}\colon 1\le j\le N\}$, where $0\le t_{j} < 2 \pi$, $t_j\neq t_k$ for $j\neq k$. In this case,
    \[
 q(z) = \sum_{j=1}^{N} r_{j}\,\frac{1+\a_{j}z}{1-\a_{j}z}, \qquad \sum_{j=1}^{N} r_{j} = 1.
    \]

For arbitrary $q\in\cp$, by Carath\'eodory's lemma \cite[p.~41]{D}, we know that $|q_n| \leq 2$ for all $n\ge 1$, with equality only if the support of the measure is a subset of the $n$-th roots of some unimodular complex number. Livingston \cite{L} further showed that $|q_n-q_{k}q_{n-k}| \leq 2$ for all $0 \leq k \leq n$, and this was later generalized in several directions by Efraimidis \cite{Ef}.
\par\smallskip
	Following \cite{Ef}, for a complex number $w$ and a function $q$ in $\cp$ we consider the $(k+1) \times (k+1)$ determinants
	
	\begin{equation*}
		A_{k,m}(w)= \left|
		\begin{array}{cccccc}
			q_{m+k} & q_{m+k-1} & q_{m+k-2} &  \ldots & q_{m+1} & q_m   \\
			wq_1     & 1               & 0               &  \ldots & 0            & 0       \\
			wq_2     & wq_1         & 1               &  \ldots & 0            & 0       \\
			\vdots    & \vdots        &   \vdots      & \ddots & \vdots     & \vdots\\
			wq_{k-1}& wq_{k-2}   & wq_{k-3}   &  \ldots & 1            &   0      \\
			wq_k     & wq_{k-1}   & wq_{k-2}    &  \ldots & wq_1      &   1
		\end{array} \right|.
	\end{equation*}

\smallskip\noindent
	The following was proved by Efraimidis in \cite[Theorem 2]{Ef}.
\par
	\begin{lem}\label{lem:Efraimidis}
		For $q$ in $\cp$ and $w$ in $\C$ we have
		\[|A_{k,m}(w)| \leq 2 \max\{1, |1-2w|^k\},\]
		for all integers $k \geq 0$ and $m \geq 1$.
\par
    If $|1 - 2w| < 1$, then equality holds if and only if the support of the Herglotz measure of $q$ is contained in a rotation $\{\z: \z^{m+k}=\la\}$, $|\la|=1$, of the set of $(m+k)$-th roots of unity and $q_1 = q_2 = \ldots = q_k = 0$.
	\end{lem}
    In what follows, we will use Efraimidis' estimates to prove Theorem~\ref{thm:Gn}. We will also need the classical Schur complement formula for determinants; see \cite[p.~217]{Sch}.
	\begin{lem} \label{lem:Schur}
		Let $M$ be a square matrix of complex numbers written in block form
		\[
			M = \begin{pmatrix}
				A & B \\
				C & D
			\end{pmatrix},
		\]
		where $A$ and $D$ are square matrices of sizes $m\times m$ and $k\times k$, respectively.
		If $D$ is invertible, then
		\[
			\det M = \det D \cdot \det(A - B D^{-1} C).
		\]
	\end{lem}
	\begin{proof}
		The identity follows by taking determinants in the factorization
			\[
				M \begin{pmatrix}
					I_m & 0 \\
					-D^{-1}C & I_k
				\end{pmatrix} = \begin{pmatrix}
					A - BD^{-1}C & B \\
					0 & D
				\end{pmatrix}.
			\]
	\end{proof}

	\begin{lem}\label{lem:det1}
		Fix an integer $n \geq 6$ and let $l=\lfloor n/3 \rfloor$.
		Let $q$ be a function in $\cp$, and let $w$ be any complex number. Assume that $q_j=0$ for all $1 \leq j \leq l$. Then
		\begin{equation}\label{eq:A1}
			\left|q_n-w\sum_{j=l+1}^{n-1-l}q_{n-j}q_{j} \right| \leq 2 \max\{1, |1-2w|^{n-l-1}\}.
		\end{equation}
\par
		If $|1 - 2w| < 1$, then equality holds if and only if the support of the Herglotz measure of $q$ is contained in a rotation $\{\z: \z^{n}=\la\}$, $|\la|=1$, of the set of $n$-th roots of unity and $q_1 = q_2 = \ldots = q_{n-1-l} = 0$.
	\end{lem}
	\begin{proof}
		Since $n \geq 6$, observe that $n - 1 - l \geq l + 1 \geq 3$. The idea of the proof is to calculate the determinant $A_{n-1-l,l+1}(w)$ and then apply Lemma \ref{lem:Efraimidis}. If $n = 6$, then \eqref{eq:A1} becomes
			\[
				|q_6 - w q_3^2| \leq 2 \max \{1, |1 - 2w|^3 \}.
			\]
			Note that
			\[
				A_{3,3} (w) = \begin{vmatrix}
					q_6 & q_5 & q_4 & q_3 \\
					0 & 1 & 0 & 0 \\
					0 & 0 & 1 & 0 \\
					wq_3 & 0 & 0 & 1
				\end{vmatrix} = q_6 - w q_3^2,
			\]
			since $q_1=q_2=0$, so there is nothing left to do.
		
		Assume now that $n \geq 7$. We see that
			\[
				A_{n - 1 - l, l + 1} (w) = \det \begin{pmatrix}
					q_n & R \\ C & T
				\end{pmatrix},
			\]
			where $R = \begin{pmatrix}
				q_{n - 1} & q_{n - 2} & \ldots & q_{l+1}
			\end{pmatrix}$, $C^T = w \begin{pmatrix}
				q_1 & q_2 & \ldots & q_{n - l - 1}
\end{pmatrix}$ and the $(n - 1 - l) \times (n - 1 - l)$ matrix
			\[
				T = \begin{pmatrix}
					1 & 0 & 0 & \ldots & 0 \\
					wq_1 & 1 & 0 & \ldots & 0 \\
					wq_2 & wq_1 & 1 & \ldots & 0 \\
					\vdots & \vdots & \vdots & \ddots & \vdots \\
					w q_{n - 2 - l} & w q_{n - 3 - l} & w q_{n - 4 - l} & \ldots & 1
				\end{pmatrix}.
			\]
			Observe that $T$ is a triangular Toeplitz matrix and thus we may write
			\[
				T = I + w \sum_{j = 1}^{n - 2 - l} q_j J^j,
			\]
			where $I$ is the identity matrix of size $n - 1 - l$ and $J$ is the matrix of the same size whose entries are equal to $1$ on the first subdiagonal and to $0$ elsewhere. Since $n \geq 7$, then $n - 2 - l \geq l + 1$ and the hypothesis $q_1 = q_2 = \ldots = q_l = 0$ implies that
			\[
				T = I + w \sum_{j = l+1}^{n - 2 - l} q_j J^j.
			\]
			
			By definition of $l$, there exists a unique $s \in \{0, 1, 2 \}$ such that $n = 3l + s$. Thus, $n - 1 - l = 2l - 1 + s \leq 2(l + 1)$ and, since $J^{n - 1 - l} = 0$, we have that
			\[
				\left( w \sum_{j = l+1}^{n-2-l} q_j J^j\right)^2 = 0.
			\]
			Summing up, we have that $T$ must be invertible and
			\begin{equation} \label{eq:InverseToeplitz}
				T^{-1} = I - w \sum_{j = l+1}^{n-2-l} q_j J^j.
			\end{equation}
			
			Recall that, from the hypothesis, the first $l$ components of $C$ are zero. Since $J^jC$ is obtained from $C$ by shifting its entries $j$ positions downwards and $2l+1\ge n-1-l$, it follows that $J^j C = 0$ for all $l + 1 \leq j \leq n - 2 - l$ and thus $T^{-1}C = C$ by \eqref{eq:InverseToeplitz}. Then Lemma~\ref{lem:Schur} yields that
			\[
				A_{n-l-1,l+1} (w) = \det T \cdot \det(q_n - R T^{-1} C) = q_n - RC = q_n - w \sum_{j = l + 1}^{n - l - 1} q_{n - j} q_j.
			\]	
		(recall that $\det T=1$). The proof now follows by applying Lemma~\ref{lem:Efraimidis} with $k=n-1-l$ and $m=l+1$.
	\end{proof}
\par
		Next, we show that the bound on the right-hand side of \eqref{eq:A1} can be improved when $w\ge 1$ is real.
		\begin{lem} \label{lem:ImprovementDet1}
			Under the assumptions of Lemma~\ref{lem:det1}, if $w \geq 1$ is real, then
			\[
				\left| q_n - w \sum_{j = l+1}^{n - 1 - l} q_{n-j} q_j \right| \leq 2(2w-1).
			\]
		\end{lem}
		\begin{proof}
			For $w=1$, \eqref{eq:A1} is exactly what we want to show. If $w > 1$, we can apply the case $w=1$, Carath\'eodory's lemma, and the triangle inequality to see that
			\[
			\left| q_n - w \sum_{j = l+1}^{n - 1 - l} q_{n-j} q_j \right| = \left| -(w-1) q_n + w\left(q_n - \sum_{j = l+1}^{n - 1 - l} q_{n-j} q_j \right)\right| \leq 2(w-1) + 2w = 2(2w-1).
			\]
		\end{proof}

\par
\subsection{Equivalent formulations of the generalized conjecture}
 \label{subsect:equiv-form}
Since the Krzy\.z conjecture is believed to be true by many experts, we expect the following more general conjecture for problem \eqref{eq:extr-alt-p} to be true.
\begin{conj}
Let $n\ge 1$ and let $\F_p$ be an admissible functional. Then $G_n(p)=2/e$ and the only extremal function for \eqref{eq:extr-alt-p} is $f_n(z)= e^{\frac{z^n-1}{z^n+1}}$.
\end{conj}
\noindent
By Corollary~\ref{cor:small-n}, the conjecture holds for $n\le 5$, and by Theorem~\ref{thm:rel-probl}~\ref{thm:rel-probl:a}, it holds for every $n$ for which the Krzy\.z conjecture is true. This raises the natural question as to whether some parts of the main theorem from \cite{MSUV} (see Theorem~\ref{thm:A}) can also be formulated in this more general context.
\par
For a non-constant function $f$ in $\cb_\ast$ we may write $f=e^g$ where $\re g (z) <0$ for all $z$ in $\D$. Writing $g(z) = \sum_{j=0}^{\infty} b_j z^j$ and assuming that $b_0$ is real and negative, it is immediate that the normalized function $q$ given by $q(z)=g(z)/b_0$ belongs to $\cp$; in particular, $|b_n| \leq 2|b_0|$ for all $n\ge 1$ by Carath\'eodory's lemma. The relation between the coefficients of $f$ and $g$ is expressed by the following well-known result; \textit{cf.\/} \cite{MSUV}, for example.
	\begin{lem}\label{lem:exp-depend}
Let $f$ and $g$ be analytic functions in $\D$ such that $f=e^g$, given by their respective power series $f(z)=\sum_{j=0}^{\infty} a_j z^j$ and $g(z)=\sum_{j=0}^{\infty} b_j z^j$. Then
	\begin{enumerate}[(a)]
\item \label{lem:exp-depend:a} $a_0=e^{b_0}$ and
		\[		 a_n=\sum_{k=0}^{n-1}\frac{n-k}{n}a_kb_{n-k}=\sum_{j=1}^{n}\frac{j}{n}a_{n-j}b_j, \quad n \geq 1.
		\]
\item \label{lem:exp-depend:b} For a fixed integer $j\ge 1$, $b_1=b_2=\dots=b_j=0$ holds if and only if $a_1=a_2=\dots=a_j=0$.
    \end{enumerate}
	\end{lem}
\par
The following formula was described structurally in \cite[Proposition~1]{MSUV} but had actually been known earlier to Lewandowski and Szynal \cite[Lemma~1]{LS}, whose detailed survey may not have received all the attention it deserved. The formula was stated there in a more precise form as follows.

	\begin{lem}\label{lem:coeff-formula}
		Let $f$ and $g$ be analytic functions in $\D$ such that $f=e^g$, given by their respective power series $f(z) = \sum_{j=0}^{\infty} a_j z^j$ and $g(z)=\sum_{j=0}^{\infty} b_j z^j$. Assuming that $b_0$ is real and negative and writing $t = -b_0 > 0$ and $q_k = b_k/b_0$ for $k\ge 1$, we have
\[
    a_n = e^{-t} \sum_{k_1 + 2k_2 + \ldots + nk_n = n, \, 0 \leq k_j} \dfrac{(-t)^{k_1 + k_2 + \ldots + k_n} q_1^{k_1} q_2^{k_2} \ldots q_n^{k_n}}{k_1! k_2! \ldots k_n!}.
    \label{eq:coeff-formula}
\]
	\end{lem}

We now have all the ingredients needed to prove Theorem~\ref{thm:Gn}, formulated in Section~\ref{sect:intro}.

	\begin{proof}[Proof of Theorem~\ref{thm:Gn}]
One implication is trivial, hence it suffices to prove that \ref{thm:Gna} implies \ref{thm:Gnb}.

Assume that \ref{thm:Gna} holds. Since $f\in\cb_n$, it is a non-constant function in $\cb_\ast$ with $f(0)=a_0>0$, so it can be written as $f=e^g$, where $g(z)=\sum_{j=0}^\infty b_j z^j$ is analytic in $\D$, $\re g<0$, and $b_0=\log a_0$ is real and negative. Writing again $t = -b_0 > 0$ and $q_k = b_k/b_0$ for $k\ge 1$ as in Lemma~\ref{lem:coeff-formula}, we have $g = -t g_0$, where $g_0\in \cp$ and $g_0(z) = 1 + \sum_{j \geq 1} q_j z^j$.
		
		By our assumption \ref{thm:Gna} and Lemma~\ref{lem:exp-depend}~\ref{lem:exp-depend:b}, we have $b_1=\ldots=b_l=0$, that is, $q_1=\ldots=q_l=0$, so $g_0 - 1$ has a zero of order at least $l + 1$ at the origin. Hence, Lemma~\ref{lem:coeff-formula} yields
			\begin{equation}
				a_n = e^{-t} \sum_{(l+1) k_{l+1} + \ldots + nk_n = n, \, 0 \leq k_j} \dfrac{(-t)^{k_{l+1} + \ldots + k_n} q_{l+1}^{k_{l+1}} \ldots q_n^{k_n}}{k_{l+1}! \ldots k_n!}.
            \label{eq:LSRestricted}
			\end{equation}
			We will show that $(l+1) k_{l+1} + \ldots + n k_n = n$ for an $(n-l)$-tuple of indices $(k_{l+1}, \ldots, k_n)$ if and only if either $(k_{l+1}, \ldots, k_n) = (0, \ldots, 0, 1)$ or there exists $j \in \{l+1, \ldots, n - 1 - l \}$ such that $k_j = k_{n-j} = 1$ if $j\neq n-j$ (respectively, $k_{n/2}=2$ if $j=n/2$) and all the remaining indices are zero. These are exactly the solutions with $\sum_{j = l+1}^n k_j \leq 2$; if $\sum_{j = l+1}^n k_j \geq 3$, we would have that
			\[
				n = \sum_{j = l+1}^n j k_j \geq (l+1)  \sum_{j = l+1}^n k_j \geq 3(l+1),
			\]
			which is impossible by the choice of $l$. Then \eqref{eq:LSRestricted} becomes
			\[
				a_n = -te^{-t} \left( q_n - \dfrac{t}{2} \sum_{j=l+1}^{n-1-l} q_j q_{n-j}  \right)
			\]
			(note that every product $q_j q_{n-j}$ with $j\neq n-j$ appears twice in the sum, while for even $n$ the term $q_{n/2}^2$ appears once and comes with the factor $1/2!$ in \eqref{eq:LSRestricted}). We may apply Lemma~\ref{lem:det1} for $0<t\le 2$, with $w=t/2$ (note that $|1-2w|=|1-t|\le 1$) and Lemma~\ref{lem:ImprovementDet1} for $t\ge 2$), also with $w=t/2$, to conclude that
			\[
				a_n = |a_n| \leq 2t e^{-t} \max \{1, t - 1\}.
			\]
			A straightforward computation shows that
            \[
\max \{te^{-t} : 0 < t \leq 2 \} = e^{-1} > (2 + \sqrt{5}) e^{-(3+\sqrt{5})/2} = \max \{t (t-1) e^{-t} : t \ge 2 \}.
            \]
The maximum of the expression on the left is only attained when $t = 1$, while the maximum on the right is attained only at $t=(3+\sqrt{5})/2$. Thus, $a_n \leq \dfrac{2}{e}$, and $G_n(p) = \frac{2}{e}$ readily follows from the chain of inequalities (the first one holds since $f_n\in\cb_n$ and $\F_p(f_n)=2/e$)
            \[
 \frac{2}{e} \leq G_n(p) = \Phi_p (f) \leq a_n.
            \] 		
			Moreover, $a_n=2/e$, so equality must hold in all the estimates above; in particular, $t=1$ and equality holds in \eqref{eq:A1} with $w=\frac12$.
\par
To finish the proof, it remains to show that $f=f_n$. Since $f\in\cb_n$ and $t=1$, we can write $g_0(z) = \sum_{j = 1}^N r_j \frac{1 + \alpha_j z}{1 - \alpha_j z}$ for certain $N\le n$, $r_j > 0$ with $\sum_{j=1}^N r_j=1$, and distinct points $\alpha_j \in \T$. Since equality holds in \eqref{eq:A1} for $g_0$ with $w = \frac{1}{2}$, Lemma~\ref{lem:det1} shows that $q_1 = q_2 = \ldots = q_{n - 1 - l} = 0$ (hence, $b_1 = b_2 = \ldots = b_{n - 1 - l} = 0$) and that there is $\lambda\in\T$ with $\alpha_j^n = \lambda$ for all $1 \leq j \leq N$. The power series expansion of $\frac{1 + \alpha_j z}{1 - \alpha_j z}$ yields
		\[
			b_{n - k} = - 2 \sum_{j = 1}^N r_j \alpha_j^{n-k} = \lambda \overline{b_k}, \quad 1 \leq k \leq n-1,
		\]
hence $b_{l+1} = \ldots = b_{n-1} = 0$ as well. Now Lemma~\ref{lem:exp-depend}~\ref{lem:exp-depend:b} readily implies that $a_1 = \ldots = a_{n-1} = 0$. Once we know this, the fact that $f(z) = e^{\frac{z^n-1}{z^n+1}}$ is proved by following \textit{verbatim\/} the proof of the implication (b) $\Longrightarrow$ (c) of \cite[Theorem~1]{MSUV}, based on a simple use of a system of linear equations and the properties of the Vandermonde determinant.	
	\end{proof}
\par
As already mentioned, the case $p\equiv 0$ of Theorem~\ref{thm:Gn} is the following result of MacDonald \cite{MD}; our proof of it is independent of the one in \cite{MD}.
	\begin{cor} \label{cor:MacDonald}
	Let $n\ge 6$ and let $f(z)=\sum_{j=0}^\infty a_j z^j$ be a normalized extremal function for the Krzy\.z problem \eqref{eq:extr-ref}. If $a_1=\ldots=a_{\lfloor n/3\rfloor}=0$, then $M_n=2/e$ and $f(z)=e^{\frac{z^n-1}{z^n+1}}$.
	\end{cor}
	\begin{proof}
	By Theorem~\ref{thm:A}, $f\in\cb_n$ and $f$ is an extremal function for \eqref{eq:extr-alt-p} with $p\equiv 0$, for which $G_n(0)=M_n$. Thus, the statement is the case $p\equiv 0$ of Theorem~\ref{thm:Gn}.
	\end{proof}

It should be noted that if the recent proof \cite{W} is confirmed for $n=6$, then this case in Theorem~\ref{thm:Gn} will no longer represent a novelty.

\subsection*{AI disclosure}
    This project was started in 2024, well before the most recent advances in AI; AI tools were used only in the final stages of our work.

    ChatGPT~5.6~Sol was used to simplify our original proof of Lemma~\ref{lem:det1}. More precisely, it suggested the use of the Schur complement formula (Lemma~\ref{lem:Schur}) to simplify the notation. The original statement of Theorem~\ref{thm:Gn}~\ref{thm:Gna} included the extra assumption $a_0 \ge 1/e^2$. ChatGPT~6~Astra provided the observation to be found in Lemma~\ref{lem:ImprovementDet1}, which was then used to remove this extra condition. Example~\ref{ex:more-atoms} in Section~\ref{sect:spec-probl} was also suggested by ChatGPT~5.6~Sol.

    Claude Opus~5.5 (Anthropic) and  ChatGPT~6~Astra were used to proofread the manuscript. The former suggested various useful  simplifications, improvements in exposition, and observations such as Remark~\ref{rem:total-mass}.

\subsection*{Acknowledgments}
    The three authors are partially supported by PID2025-168349NB-I00 from Mi\-nisterio de Innovaci\'on, Ciencia y Universidades, Spain.

    This work was partially supported by the project Pure Mathematics in Norway, funded by Trond Mohn Foundation and Troms\o{} Research Foundation. This work started when the first author was visiting Universidad Aut\'onoma de Madrid and the first author would like to thank the university for the hospitality as well as Norges tekniske h\o{}gskoles fond for covering parts of this travel.

    The second author is partially supported by PID2024-160326NA-I00, also from MICIU, and by CIGE/2026/195 from Conselleria d’Edu\-caci\'o, Cultura i Universitats, Generalitat Valenciana, Spain.


\end{document}